\documentclass[12pt,a4article]{article}

\usepackage{enumerate}
\usepackage[latin1]{inputenc}
\usepackage{amsthm, amssymb, amscd,amsmath, amsfonts}
\usepackage{color}
\usepackage{graphicx}
\usepackage{hyperref}
\usepackage{authblk}
\usepackage{geometry}
\usepackage[T1]{fontenc}
 \usepackage[titletoc,toc,title]{appendix}

\usepackage{color,soul}
\usepackage{bigints}
\usepackage{fancyhdr}

\usepackage{hyperref}

\usepackage{amssymb}

\usepackage{mathrsfs,mathtools}
\usepackage[english]{babel}

\usepackage{amsmath,amsthm, bbm}

\usepackage{dsfont}
\usepackage{varioref}

\hypersetup{colorlinks=true, citecolor=blue, urlcolor=blue,linkcolor=blue}

\newtheorem{definition}{Definition}

\newtheorem{theorem}{Theorem}
\newtheorem{lemma}[theorem]{Lemma}

\newtheorem{remark}[theorem]{Remark}

\newtheorem{propertie}[theorem]{Properties}

\def\R{\mathbb{R}}

\def\N{\mathbb{N}}

\def\E{\mathbb{E}}

\def\to{\rightarrow}

\def\lto{\longrightarrow}

\def\bv{\big\vert}

\def\1{\mathbbm{1}}
\def\e{\varepsilon}

\author{A. Lechiheb and E. Haouala}

\title{Fluctuations in 1D stochastic homogenization of Pseudo-elliptic equations with Long-range dependent potentials}
\begin{document}

\maketitle
\begin{abstract}
This paper deals with the homogenization problem of one-dimensional pseudo-elliptic equations with a rapidly varying random potential.
The main purpose is to characterize the homogenization error (random fluctuations), i.e., the difference between the random solution and the homogenized solution, which strongly depends on the autocovariance property of the underlying random potential. It is well known that when the random potential has short-range dependence, the rescaled homogenization error converges in distribution to a stochastic integral with respect to standard Brownian motion. Here, we are interested  in potentials with long-range dependence and we prove convergence to stochastic integrals with respect to Hermite process.
\end{abstract}
\textbf{Keywords}: Pseudo-elliptic equations, fluctuation theory, random homogenization, long-range dependence, Hermite processes.
\section{Introduction}
We consider the following one-dimensional pseudo-elliptic equation:
\begin{equation}\label{EP1}
 \begin{cases}
 \displaystyle{P(x,D)u_\varepsilon(x,\omega) \,+\,  \left(q_0(x) +q(\frac{x}{\varepsilon},\omega)\right) u_\varepsilon(x,\omega) } = f(x), \quad x\in (0,1),\\
 u_\varepsilon(0, \omega) =  u_\varepsilon(1, \omega) = 0,
 \end{cases}
 \end{equation}
where
\begin{itemize}
  \item $P(x,D)$ is a deterministic self-adjoint, elliptic and pseudo-differential operator,
  \item $x\mapsto q_0(x)$ is a smooth function bounded by some positive constant $\gamma$,
  \item $q(\frac{x}{\varepsilon},\omega)$ is the rescaling of a bounded, stationary and mean zero process $q(x,\omega)$, defined on some abstract probability space $\big( \Omega, \mathscr{F}, \mathbb{P}\big)$,
  \item $f(x)\in L^2\big((0,1), dx \big)$ is the source term.
\end{itemize}
It is well known that, under mild conditions such as stationarity and ergodicity of the random process $q(x,\omega)$, the homogenization/averaging of such problem, i.e., where the randomness appears as a potential, is obtained simply by averaging $\tilde{q}_\varepsilon(x,\frac{x}{\varepsilon},\omega)= q_0(x)+q(\frac{x}{\varepsilon},\omega)$ (see, e.g., \cite{B08,Figari,JKO94}).
Then $u_\varepsilon$ converges, for instance in $L^2\big((0,1)\times \Omega\big)$, to $u_0$, which is the unique solution to the unperturbed equation
\begin{align}\label{homog}
 \begin{cases}
 \displaystyle{(P(x,D)+q_0)u_0(x,\omega) = f(x)}, \quad x\in (0,1), \\
 u_0(0, \omega) =  u_0(1, \omega) = 0.
 \end{cases}
 \end{align}
We define the operator
\begin{equation} \label{operator}\mathcal{G}\,:=\, \big(P(x,D) +q_0\big)^{-1},\end{equation}
which is well defined almost everywhere in $\Omega$. Assume that $\mathcal{G}$, as transformation on $L^2\big([0,1]\big)$, is bounded for all realizations, and the appear bound of the operator is independent of the potential.
Finally, we assume the existence of a Green function $G(x,y)$ associated to $\mathcal{G}$, namely,
\begin{equation}\label{} u(x) \,=\, \mathcal{G}f(x) \,:=\,\int_0^1 G(x,y) f(y) \,dy, \end{equation} which has, for $\beta\in(0,1)$, a singularity of the type
\begin{equation}\label{boundG} \vert G(x,y) \vert \leq \frac{C}{\vert x-y\vert^{1-\beta}}, \end{equation} for some universal constant $C$.

The main objective of this paper is to analyse the random fluctuation (homogenization error), that is, the difference between the solution $u_\varepsilon$ and the homogenized solution $u_0$. In other words, we are going to compute the rate of convergence and characterize the limiting distribution of the rescaled fluctuations. The same question has been addressed by several authors. They have shown that the random fluctuations strongly depends on the autocovariance property of the underlying random potential. Firstly, it was shown in \cite{B08} that if $q(x,\omega)$ satisfies certain mixing assumption and has an integrable autocovariance function, in which case we say that $q(x, \omega)$ has short-range correlation, then a CLT holds:
\begin{equation}\label{}
\frac{u_\varepsilon(x)-u_0(x)}{\sqrt{\varepsilon}} \Rightarrow \int_0^1 G(x,y) u_0(y)d B(y)
\end{equation}
in $\mathcal{C}([0,1])$, where $G(x,y)$ is the deterministic Green function satisfying condition \eqref{boundG} and $B(y)$ is a standard Brownian motion. Then the result has been extended to a large class of random potentials with long-range autocovariance function that decays like $|x|^{-\alpha}$ for some $\alpha\in(0,1)$ and the homogenization error amplitude is of order $\varepsilon^{\alpha/2}$. In this case, we have
\begin{equation}\label{}
\frac{u_\varepsilon(x)-u_0(x)}{\varepsilon^{\alpha/2}} \Rightarrow \int_0^1 G(x,y) u_0(y)d B^H(y),
\end{equation}
where $B^H(y)$ is a fractional Brownian motion with Hurst index $H=1-\frac{\alpha}{2}$ (see \cite{BG12} for more details).

In this paper, we focus on \eqref{EP1} and follow the framework in \cite{BG12} to obtain an extension to limits of the random fluctuations in the presence of long-range dependence. Based on the works of \emph{Murad Taqqu} especially on \emph{Convergence of integrated processes of arbitray Hermite rank}, we construct a large class of random potentials with long-range dependence which contains, in particular, the case already dealt in \cite{BG12}. Additionally, the arguments rely strongly upon an application of Taqqu's theorem (see Theorem \ref{Taqqu79} below).
Our main result in Theorems \ref{homogenization} and \ref{Non-central-limit-theorem} says that, under a proper assumptions on the random potential, the homogenization error amplitude is of order $\varepsilon^{\theta}$, and then
\begin{equation}\label{}
\frac{u_\varepsilon(x)-u_0(x)}{\varepsilon ^\theta} \Rightarrow \int_0^1 G(x,y) u_0(y)d Z(y),
\end{equation}
where $\theta$ is some positive constant described bellow and $Z$ is the Hermite process.

The rest of the paper is organized as follows. Section \ref{section2} presents the setup and our main results. In section \ref{section3}, we recall the definition of the Hermite process and we construct the Wiener integrals with respect to it. Finally, section \ref{section4} contains the proofs of our main results.

\section{Setup and main results}\label{section2}
In this section, we will recall the concept of long-range dependence, state our main assumptions and then we will give our main results.
\subsection{Long-range dependence}
\begin{definition}
A function $L$ is slowly varying at infinity if it is positive on $[c, \infty)$ with $c\geq0$ and, for any $a>0$,
\begin{equation} \lim _{u\to\infty} \frac{L(au)}{L(u)}\,=\,1. \end{equation}
\end{definition}
For example, the two functions $L(u)\,=\, \text{Cst}>0$ and $L(u)\,=\,\log(u)$, $u>0$, are slowly varying at infinity.\\
Let $L: (0,+\infty) \to (0, +\infty)$ be a slowly varying function at $+\infty$ and $\alpha > 0$. It is known (see \cite[Proposition 1.3.6(v)]{BGT87}) that
\begin{align}\label{propp}
x^\alpha L(x) \to  + \infty \,  \quad\text{and} \quad x^{-\alpha} L(x) \to 0 \,\,,
\end{align}
as $x\to +\infty$.

The following result is known as Potter's Theorem (see \cite[Theorem 1.5.6(ii)]{BGT87}).
\begin{theorem}\label{Harry_Potter} Let $L: (0,+\infty) \to (0, +\infty)$ be a slowly varying function at $+\infty$ which is bounded away from $0$ and $+\infty$ on every compact subset of $(0,+\infty)$.  Then, for any $\delta > 0$, there exists some constant $C = C(\delta)$ such that
 $$\frac{L(y)}{L(x)} \leq C \max\Big\{ \, (\frac{x}{y})^\delta \,, \, (\frac{y}{x})^{\delta} \Big\}$$
for any $x,\, y\in(0,+\infty)$.
\end{theorem}
Having defined the notion of slowly varying functions, we turn to the definitions of long-range dependence.
\begin{definition}
A continuous-time stationary processes $\big\{X(t)\big\}_{t\in\mathbb{R}}$ is called long-range dependent (LRD) if one of the following non-equivalent conditions holds:
\begin{enumerate}[(i)]
  \item The autocovariance function of the time processes $\big\{X(t)\big\}_{t\in\mathbb{R}}$ satisfies
  \begin{equation}\label{au} \gamma_X(h) \,=\, \text{Cov}\,\big(X(h), X(0)\big) \,=\, \text{Cov}\,\big(X(h+x), X(x)\big)\,=\, L(h)h^{2d-1}\quad h,\,x\in\mathbb{R}, \end{equation} where $L$ is a slowly varying function at infinity and and $d\in(0,\frac{1}{2})$.
  \item The autocovariances of the time processes $\big\{X(t)\big\}_{t\in\mathbb{R}}$ are not absolutely integrable, that is,
\begin{equation}
\int_\mathbb{R}\vert \gamma_X(h)\vert \,dh =\infty.
\end{equation}
\end{enumerate}
\end{definition}
In many instances, for example, in \cite{BG12}, the slowly varying function in \eqref{au} is such that $L(u) \sim \text{Cst}>0$ at infinity, and then the condition (i) becomes
$$\gamma_X(h) \sim \text{Cst}\, h^{2d-1},\quad d\in(0,\frac{1}{2}).$$

\subsection{Description of the random potential}
Proceeding as in \cite{BG12}, we assume that the random potential $q$ has the following form:
\begin{equation}\label{q}
q(x,\omega) = \Phi\big( g(x, \omega) \big)  ,
\end{equation}
where the stochastic process $\big\{ g(x) \big\}_{x\in\R_+}$ and
the function $\Phi:\R\to\R$ are constructed as follows.\\
\textbf{Assumptions on $g$}.\\
\begin{itemize}
 \item Let $m\in\N^\ast$ be fixed. Let $H_0\in( 1 - \frac{1}{2m}, 1)$ and set $H = 1 + m(H_0 - 1)\in (\frac{1}{2}, 1)$,
 \item Fix  a slowly varying function $L: (0,+\infty)\to (0,+\infty)$  at $+\infty$. Assume, furthermore, that $L$ is bounded away from $0$ and $+\infty$ on every compact subset of $(0,+\infty)$. (See \cite{BGT87} for more details on slowly varying functions.)
 \item  Let $e:\R\to \R$ be a square-integrable function
   such that
   \begin{itemize}
     \item [(3a)] $ \int_\R e(u)^2 \, du = 1 $,
     \item[(3b)] $|e(u)|\leq C u^{H_0-3/2}L(u)$ for almost all $u>0$ and some absolute constant $C$,
     \item[(3c)] $e(u) \sim C_0 u^{H_0 -3/2} L(u)$, where $C_0 = \big( \int_0^\infty (u + u^2 )^{H_0-3/2} \, du \big)^{-1/2}$,
     \item[(3d)] there exists $0<\gamma < \min\big\{  H_0 - (1 - \frac{1}{2m} ) ,  1 - H_0    \big\}$  such that
      \begin{equation*} \int_{-\infty}^0 \vert e(u) e(xy+u)\vert \,du = o(x^{2H_0-2} L(x)^2) y^{2H_0 - 2 - 2\gamma} \end{equation*} as $x\to\infty$, uniformly in $y\in(0, t]$ for each given $t > 0$.
      \end{itemize}
 \item[4.] Finally, let  $W$ be a two-sided Brownian motion.
\end{itemize}
Bearing all these assumptions in mind, we can now set, for $x\in\R_+$,
 \begin{equation}\label{g}
 g(x) := \int_{-\infty}^\infty e(x- \xi) dW_\xi .
 \end{equation}
 \begin{remark}
 {\rm  One can verify (see, e.g., \cite[Section \S 2]{Taqqu79}) that the moving-average process $g$ is a stationary centered Gaussian process that exhibits a long-range dependence. More precisely, its autocovariance function $\gamma_g$ exhibits the following asymptotic behaviour:
 \begin{align}
  \gamma_g(x) := \text{Cov}(g(s),\,g(s+x))\, = \,\E\big[ g(s) g(s+x) \big] \sim  C( H_0)  x^{2H_0-2} L^2(x)\quad\mbox{ as $x\to+\infty$},
  \end{align}
where
 $$C( H_0) : =  \int_{-\infty}^\infty (u + u^2)^{H_0 - 3/2}\, du . $$
}
\end{remark}
\textbf{Example:} Fractional Gaussian Noise\\
Let $\big\{B^{H_0}(\xi)\big\}_{\xi\geq0}$ be a fractional Brownian motion with Hurst index $H_0\in(0,1]$. It is a centered continuous Gaussian process with covariance function
$$\mathbb{E}\big[B^{H_0}(\xi_1)B^{H_0}(\xi_2)\big] \;=\, \frac{1}{2} \big(|\xi_1|^{2H_0} +|\xi_2|^{2H_0}- |\xi_2-\xi_1|^{2H_0} \big).$$
The increments of $B^{H_0}(\xi)$ are stationary but not independent for $H_0 \neq \frac{1}{2}$, that is, for all $h>0$,
$$\big(B^{H_0}(\xi+h) - B^{H_0}(\xi)\big)_{\xi\geq0} \overset{\text{law}}{=}\,\, \big(B^{H_0}(\xi)\big)_{\xi\geq0}.$$
Moreover, $B^{H_0}(\xi)$ admits the following spectral representation:
\begin{eqnarray*}
  B^{H_0}(\xi) &=& A_{1,H_0} \Big( \int_{-\infty}^\infty dB(t) \int_0^\xi  (s- t)^{H_0 - 3/2}{\bf 1}_{(t_i < s)} \, ds\Big)\\
   &=& \frac{A_{1,H_0}}{H_0-\frac{1}{2}}\Big( \int_{-\infty}^\xi \left((\xi- t)^{H_0 - 1/2}-(-t)^{H_0 - 1/2}\right)\,dB(t)+\int_0^t (\xi- t)^{H_0 - 1/2}\, dB(t)\Big),
\end{eqnarray*}
where $B$ represents the standard real-valued Brownian motion and
$$A_{1,H_0}:= \left\{ \,\, \frac{ H_0 \big( 2H_0 - 1\big)}{\Big( {\displaystyle\int_0^\infty (u + u^2 )^{H_0 - 3/2} \, du}  \Big)}  \,\, \right\} ^{1/2}.$$

Now, we define the fractional Gaussian noise process as
\begin{equation}\label{FGN}
g_1(x) \,=\,  B^{H_0}(x)-  B^{H_0}(x-1), \quad x\in\mathbb{R}.
\end{equation}
This process  can be expressed as
\begin{equation}\label{}
  g_1(x) \,=\, \int_{-\infty}^\infty e(x- t) dB(t),
\end{equation}
with
$$ e(u) = \frac{1}{\sigma}\cdot\begin{cases} 0 & \text{if\,\,\, $u\leq0$},\\u^{H_0-1/2} &\text{if\,\,\, $0\leq u\leq 1$},\\ u^{H_0-1/2}-(u-1)^{H_0-1/2} &\text{if \,\,\,$u\geq 1$}, \end{cases}
$$
where
\begin{equation}\label{}
 \displaystyle \sigma = \frac{A_{1,H_0}}{H_0-\frac{1}{2}}.
\end{equation}
We can verify that the process $g_1$ satisfies all conditions cited above by choosing a slowly varying function $L$ as follows:
$$ L(u)= \begin{cases} u &\text{if\,\,\,  $0<u\leq 1$},\\ u^{3/2-H_0} \big(u^{H_0-1/2} -(u-1)^{H_0-1/2}\big)&\text{if\,\,  $u \geq1$}.\end{cases}$$

Now, we will concentrate our attention to the assumptions on the deterministic function $\Phi$.  Let $\nu$ denote the standard Gaussian measure on $\R$. Recall that every $f\in L^2(\R, \nu)$ admits the following series expansion:
 \begin{align}\label{chaotic}
  f = \sum_{q=0}^\infty  \frac{V_q}{q!} H_q,  \quad\text{with $V_q:=  \int_\mathbb{R} f(x)  H_q(x) \nu(dx)$,}
 \end{align}
 where $$H_q(x)= (-1)^q \exp(\frac{x^2}{2})\frac{d^q}{dx^q} \exp(-\frac{x^2}{2})$$
 denotes the $q$-th Hermite polynomial. Recall that the integer $m_f : = \inf\{q\geq 0\,:\,V_q\neq 0\}$ is called the {\it Hermite rank} of $f$ (with the convention $\inf\emptyset = +\infty$). For any integer $m\geq 1$, we define $\mathscr{G}_m$ to be the collection of all square-integrable functions (with respect to the standard Gaussian measure on $\R$) that have Hermite rank $m$.\\
\textbf{Assumptions on $\Phi$.} Throughout the rest of this paper, we assume that $\Phi\in \mathscr{G}_m$, satisfying \begin{equation}\label{AA1}\vert \Phi\vert\leq \gamma\leq q_0.\end{equation} The above bound of $\Phi$ ensures that $q_0+q_\varepsilon$ is non-negative, and then the operator
\begin{equation}\mathcal{G}_\varepsilon := \big(P(x,D) +q_0+ q_\varepsilon\big)^{-1}\end{equation}
is well defined which implies that equation \eqref{EP1} is well-posed almost surely.

\subsection{Asymptotic behaviour of the autocovariance function of $q$}
Now, set $\gamma_q(x) = \E\big[ q(0)q(x)\big]$, $x\in\R$, and recall that $m$ is the Hermite rank of $\Phi$. Then we have the following result.
\begin{lemma}
The autocovariance function of $q$ satisfies:
\begin{equation}\label{correlation}
\big\vert \gamma_q(x) \big\vert = \big(  o(1) +   \frac{ V_m^2}{m!} \big)  C(\sigma, H_0)^m   L(\vert x \vert)^{2m}  \vert x\vert^{-2m(1-H_0)} \, ,\text{as $\vert x \vert \to +\infty$}.
\end{equation}
Here $o(1)$ means the term converges to zero when $x\to \infty$.
\end{lemma}

\begin{proof}
Proceeding in similar lines as that of \cite[Lemma 2.1]{GB12}, since $V_0=...=V_{m-1}=0$, we have
\begin{equation} \Phi(g(x)) = \sum_{n=m}^\infty\frac{V_n}{n!}H_n(g(x)). \end{equation}
Thus
\begin{align*}
   \mathbb{E}\big[  \Phi(g(0))\Phi(g(x)) \big]  & =  \sum_{n_1,n_2=m}^\infty\frac{V_{n_1}}{n_1!}\frac{V_{n_2}}{n_2!}\mathbb{E}\big[ H_{n_1}(g(0))H_{n_2}(g(x))\big] \\
   & = \sum_{n=m}^\infty\frac{V_{n}^2}{(n!)^2}n!\gamma_g(x)^n  \\
   &= \frac{V_m^2}{m!} \gamma_g(x)^m +  \gamma_g(x)^m \cdot \sum_{n > m} \frac{V_n^2}{n!}\gamma_g(x)^{n-m} \, .
\end{align*}
It is clear that $$\sum_{n=0}^\infty\frac{V_n^2}{n!}<\infty  \qquad \text{and } \qquad  \vert \gamma_g(x) \vert \leq 1,$$
so by dominated convergence and the asymptotic behaviour of $\gamma_g$, we have
\begin{align*}
\big\vert \gamma_q(x)  \big\vert &=  \big\vert \mathbb{E}\big[ q(0)q(x)\big] \big\vert\\
                            &= \big(  o(1) +    \frac{V_m^2}{m!} \big)  C(\sigma, H_0)^m   L(\vert x \vert)^{2m}  \vert x\vert^{-2m(1-H_0)} \, ,
\end{align*}
as $\vert x \vert \to +\infty$, and the proof is now complete.
\end{proof}

The asymptotic relation \eqref{correlation} implies  the existence of some absolute  constant $C$ satisfying
\begin{align}\label{correlation_particular}
\bv \gamma_q(x) \bv \leq C\, L( \vert x \vert )^{2m}  \vert x \vert^{-2m(1-H_0)}
\end{align}
for any $x\neq 0$.

\subsection{Our main results}
Considering the assumptions above, we are ready to state the main results of this paper. The first theorem concerns the homogenization of \eqref{EP1}.
\begin{theorem}\label{homogenization}
Fix an integer $m\geq 1$ and a real number $H_0\in(1-\frac{1}{2m},1)$. Assume that $q=\{q(x)\}_{x\in\mathbb{R}_+}$ is constructed as in \eqref{q} so that $\{g(x\}_{x\in\mathbb{R}_+}$ is the Gaussian process given
by \eqref{g} and the function $\Phi$ belongs to $\mathscr{G}_m$ and satisfies \eqref{AA1}. Let $u_\e$ be the solution to \eqref{EP1}, let $u_0$ be the solution to the homogenized equation \eqref{homog} and let $f\in L^2((0,1))$. Then, assuming $2\beta<1$, we have

\begin{equation}
\mathbb{E}\|u_\e - u_0\|^2 \,\leq \|f\|^2 \times \begin{cases}
C \e^{2m(1-H_0)}, & {2m(1-H_0)}<2 \beta,\\
C \e^{2\beta}|\log \frac{1}{\e}|, & {2m(1-H_0)} = 2 \beta,\\
C \e^{2\beta} , & {2m(1-H_0)} > 2 \beta.
\end{cases}
\end{equation}
The constant $C$ depends on $m,\, H_0,\, \beta, \, \gamma$ and the uniform bound on the solution operator of \eqref{EP1}. On the other hand, since $0<2m(1-H_0)<1$, if $2\beta\geq 1$, then only the result on the first line above holds.
\end{theorem}
It follows from this theorem that the homogenization error is of order $\e^{m(1-H_0)}$. The question that now arises is the limiting distribution of the rescaled error that is  \begin{equation}\label{eee}\frac{u_\e-u_0}{\e^{m(1-H_0)}}.\end{equation}
\begin{remark}Recall that $H=1+m(H_0-1)$ and let \begin{equation}\label{d}d(x) := \sqrt{\frac{m!}{H(2H-1)}} x^H L(x)^m. \end{equation} We set
\begin{equation}\label{X}X(\e) = \e d(\frac{1}{\e}) =\sqrt{  \frac{m!}{(1+m(H_0-1))(1+2m(H_0-1))} }  \e^{m(1-H_0)}L(1/\e)^m.\end{equation}
It is well known, from properties \eqref{propp} of slowly varying function, that
\begin{align} \label{SL_fact1}
\lim_{\e\downarrow 0} \e^{m(1-H_0)} L(\frac{1}{\e})^m = 0 \,\, .
\end{align}
Then, due to the boudness of $L$, one can work with $X(\e)$ instead of $\varepsilon^{m(1-H_0)}$ on the denominator of \eqref{eee}.
\end{remark}
As we will see later, this choice allows us to apply certain approaches in order to achieve our next results.

Before stating the main theorem concerning the limiting distribution of the random fluctuations, we need another assumption on $\Phi$.\\
\bigskip
{\it \# \underline{\textbf{More assumptions on $\Phi$}.}}\\
The function $\Phi$ satisfies
\begin{equation}\label{3}
  \int_\R \vert \hat{\Phi}(\xi)\vert \big(1+ |\xi|^3\big) \,d\xi<\infty,
\end{equation}
where $\hat{\Phi}$ denotes the Fourier transform of $\Phi$.
\begin{theorem}\label{Non-central-limit-theorem}
Let $u_\varepsilon,\, u_0,\, q(x)$ and $f$ be as in the previous theorem. Assume that the function $\Phi$ appearing in the expression of $q$ satisfies also estimate \eqref{3}. Finally, we assume that the Green function $G(x,y)$ is Lipschitz continuous in $x$ with Lipschitz constant $\text{Lip}(G)$ uniform in $y$. Then, for each $\e>0$, the random fluctuation $u^\e- u_0$ is a continuous process on $[0,1]$. Moreover, we have the following convergence in law on $C([0,1])$ endowed with the supremum norm as $\e\to 0$:
\begin{align*} \frac{u_\e(x)- u_0(x)}{X(\e) } \Longrightarrow  -\frac{V_m}{m!}  \int_0^1 G(x,y)  u_0(y)\, dZ^{(m)}_H(y) \,, \end{align*} where $Z^{(m)}_H$ is the Hermite process of order $m$ and self-similar index $H =m(H_0 - 1) + 1$ defined below in Section \ref{section3} and $X(\e)$ is given in \eqref{X}.
\end{theorem}
Theorem \ref{Non-central-limit-theorem} should be seen as an extension and unified approach of the main results of \cite{BG12}. More precisely, the case where the Hermite rank of $\Phi$ is $m=1$ and the slowly varying function $L$ is a positive constant, corresponds to \cite[Theorem 2.5]{BG12} and involves the fractional Brownian motion  in the limit. In higher dimensions, the case where $m=1$ has been studied in \cite{BG12} but for arbitrary Hermite rank $m>1$ it is usually very hard to study the convergence of random fluctuations. We should firstly find the counterpart of Hermite process in Higher dimensions.
\section{The Hermite process}\label{section3}
Before giving the proof of the mains theorems, we briefly recall some general facts about the Hermite process, the continuous version of the non-central limit theorem and the stochastic integral with respect to Hermite process.
\subsection{Taqqu's theorem/ Non-central limit theorem}
\begin{definition}
Let $k\geq 1$ be an integer and
\begin{equation}
H\in(\frac{1}{2}, 1).
\end{equation}
Set \begin{equation}H_0\,=\, 1-\frac{1-H}{k} \,\in\, (1-\frac{1}{2k}, 1),\end{equation} so that $H=1-k(1-H_0)$. The Hermite process $\big\{Z^{k}_H(t)\big\}_{t\in\mathbb{R}}$ of order $k$ and Hurst index $H$ is defined as
\begin{equation}\label{Hermite}
 Z^{(k)}_H(t)\, = \, A_{k,H_0} \Big( \int_{-\infty}^\infty B(d\xi_1)\int_{-\infty}^{\xi_1} B(d\xi_2) \cdot\cdot\cdot  \int_{-\infty}^{\xi_{k-1}} B(d\xi_k)  \int_0^t \prod_{i=1}^k (s- \xi_i)^{H_0 - 3/2}{\bf 1}_{(\xi_i < s)} \, ds   \Big),
\end{equation}
where $B(du)$ is a Gaussian random measure on $\mathbb{R}$, with Lebesgue control measure $du$, $\xi_+=\max \{\xi,0\}$ and $A_{k,H_0}$ is a normalizing constant. The Hermite process $\big\{Z^{k}_H(t)\big\}_{t\in\mathbb{R}}$ is called standard if $\mathbb{E} \big(\{Z^{k}_H(1)\big)^2 =1$.
\end{definition}
\begin{propertie}
\begin{itemize}
  \item Note that $\big\{Z^{k}_H(t)\big\}_{t\in\mathbb{R}}$ lives in the Wiener chaos of order $k$, which is non-Gaussian unless $k = 1$ or $t = 0$.
  \item The Hermite process of order one is the Fractional Brownian motion and in this case $H = H_0$. The Hermite process of order two is the Rosenblatt process.
  \item Hermite processes are well-defined, they have stationary increments and they are $H$-self-similar. The Hermite process of order $k$ is standard when
  \begin{align}\label{coeff_Hermite_defn}
 A_{k,H_0} := \left\{ \,\, \frac{k! \big[ k(H_0 - 1) + 1 \big] \big[ 2k(H_0 - 1) + 1   \big]}{\Big( {\displaystyle\int_0^\infty (u + u^2 )^{H_0 - 3/2} \, du}  \Big)^k}  \,\, \right\} ^{1/2}.
 \end{align}
\end{itemize}
\end{propertie}
Now, let $g$ be the centered stationary Gaussian process defined by (\ref{g}), and assume that $\Phi\in L^2(\R,\nu)$ has Hermite rank $m\geq 1$.
Recall $d(x)$ from \eqref{d}. The main property of $d(x)$ is that the variance of ${\displaystyle \dfrac{1}{d(x)}\int_0^x H_m(g(y))\, dy} $ turns out to be asymptotically equal to $1$ as $x\to +\infty$.

The following result, due to Taqqu in 1979, is the key ingredient  in our proofs.
 \begin{theorem}(\cite[Lemma 5.3]{Taqqu79})\label{Taqqu79}
As $T\to+\infty$,  the process
\begin{equation}\label{int-Taqqu}
 Y_T(x) = \frac{1}{d(T)} \int_0^{Tx}\Phi\big[ g(y) \big]  \, dy,\quad  x\in\R_+,
\end{equation}
 converges  to ${\displaystyle  \frac{V_m}{m!}  Z^{(m)}_H(x)}$ in the sense of finite-dimensional distributions, where $Z^{(m)}_H(x)$ is the standard Hermite process of order $m$ and self-similar index $H =m(H_0 - 1) + 1$.
 \end{theorem}

\subsection{Hermite Wiener Integrals and their Deterministic Integrands}

The Wiener integrals with respect to the Hermite process $Z^{(k)}_H$  on $\mathbb{R}$ are integrals of the form
\begin{equation}\label{}
\int_\mathbb{R} f(u)\,dZ^{(k)}_H(u):= \mathcal{I}^k_H(f),\end{equation}
where $f$ is a deterministic function. Let us recall how we classically define the Wiener integral with respect to the Brownian motion: first we define it for elementary functions and establish the isometry property, and then we extend the integral for general functions via isometry.  In the same way, one can define the Wiener integral with respect to the Hermite process $Z^{(k)}_H$, while the annoying issue is that  one should find a suitable space for the latter extension.

Let $\mathscr{E}$ be the set of elementary functions, that is, the set of all functions  $f$ of the form
 $$ f(x) = \sum_{i=1}^\ell a_i \mathbf{1}_{(t_i, t_{i+1}] }(x) $$
with $\ell\in\N^\ast$, $a_i\in\R$, $t_i < t_{i+1}$.

For such $f$, we define the Wiener integral with respect to $Z^{(k)}_H$ in the usual way, that is, as a linear functional over $\mathscr{E}$:
 $$\int_\R f(x) \, dZ^{(k)}_H(x)  =  \sum_{i=1}^\ell a_i\Big[ Z^{(k)}_H(t_{i+1}) - Z^{(k)}_H(t_i)  \Big] \,\, .  $$
One can verify easily that this definition is independent of choices of representation for elementary functions.   Now we introduce the space of (deterministic) integrands for this Wiener integral, namely,
        \begin{equation}
  \Lambda^H = \left\{ \,\, f:  \R \lto\R  \, \Big\vert\,  \quad\int_\R \int_\R f(u)f(v) \vert u-v \vert^{2H-2} \, du\, dv  <  + \infty  \,\,  \right\},
            \end{equation}
equipped with the norm
\begin{equation}
  \| f \|^2_{\Lambda^H} = H(2H-1)\int_\R\int_\R   f(u)f(v) \vert u-v\vert^{2H-2}\, du \, dv \, .
\end{equation}
When $h\in\mathcal{E}$, it is straightforward to check the following isometry property:
           $$ \E\left[\left(\int_\R h(x)dZ(x)\right)^2 \right] = \|h\|^2_{\Lambda^H}. $$
As a consequence, one can define the Wiener integral $\int_\R f(x)dZ(x)$  for any $f\in \Lambda^H$, by a usual approximation procedure.

It is well known (thanks to  \cite{PT00}) that  $\big( \Lambda^H, \| \cdot \| _{\Lambda^H} \big)$ is a Hilbert space that contains distributions in the sense of Schwartz. To overcome this problem, we shall restrict ourselves to the proper subspace
 $$\vert {\Lambda^H} \vert =  \left\{ \,\, f:\R\to\R  \, \Big\vert\,  \quad\int_\R\int_\R \vert f(u)f(v) \vert  \vert u-v \vert^{2H-2} \, du\, dv  <  + \infty  \,\,  \right\}  $$
equipped with the norm
$$   \|f\|^2_{\vert  \Lambda^H  \vert} = H(2H-1) \int_\R\int_\R \vert f(u)f(v)\vert \vert u-v\vert^{2H-2}\, du \, dv \, .$$
We then have (see \cite[Proposition 4.2]{PT00})
\begin{equation}
 L^{1} \big(\R \big)\cap L^{2} \big(\R \big)\subset L^{1/H} \big(\R \big)\subset \vert \Lambda^H \vert \subset  \Lambda^H .\label{inclusion}
\end{equation}
Moreover, $\big( \vert \Lambda^H \vert,   \| \cdot \|_{\vert \Lambda^H \vert} \big)$ is  a Banach space, in which the set $\mathcal{E}$  is dense. So for $h\in  \vert \Lambda^H \vert$, we can define
\begin{equation}\label{cv}
\int_\R h(x) \, dZ(x) = \lim_{n\to+\infty}\int_\R h_n(x) \, dZ(x) \,,
\end{equation}
where $(h_n)$ is any sequence of $\mathcal{E}$ converging to $h$ in  $\big(  \vert \Lambda^H \vert,   \|\cdot \|_{ \vert \Lambda^H \vert} \big)$; the convergence in (\ref{cv}) takes place in $L^2(\Omega\big)$.

For a detailed account of this integration theory, one can refer to \cite{MT07,PT00}.

\section{Proofs of the main results}\label{section4}
\subsection{Proof of Theorem \ref{homogenization}}
Our proof is strictly similar to that of \cite[Theorem 2.1]{BG12} but in an one dimensional space and in a general Long-range dependence setting. We start by giving the following lemma
\begin{lemma}
Let all assumptions of theorem \ref{homogenization} be satisfied. Moreover, let $\mathcal{G}$ be the operator defined in  \eqref{operator}. Then, assuming $2\beta<1$ we get:

\begin{equation}
\mathbb{E}\|\mathcal{G}q_\e f\|^2 \,\leq \|f\|^2 \times \begin{cases}
C \e^{2m(1-H_0)}, & {2m(1-H_0)}<2 \beta\\
C \e^{2\beta}|\log \frac{1}{\e}|, & {2m(1-H_0)} = 2 \beta\\
C \e^{2\beta} , & {2m(1-H_0)} > 2 \beta.
\end{cases}
\end{equation}
The constant $C$ depends only on $m, H_0, \|q\|_\infty$ and the bound for $\|\mathcal{G}_\e\|_{\mathcal{L}}$. When $2\beta\geq1$, then only the first line holds.
\end{lemma}
Note that here, since $L(x) =\log x$ is a slowly varying function at infinity, by properties \eqref{propp} we have $$\e^{2m(1-H_0)}|\log \frac{1}{\e}|\to 0,\quad \text{as}\,\,\, \e\to0.$$

\begin{proof}
The $L^2$ norm of $\mathcal{G}q_\e f$ has the following expression

$$\|\mathcal{G}q_\e f\|^2 \,=\, \int_0^1\left(\int_0^1 G(x,y) q_\e(y) f(y) \,dy\right)^2 \,dx.$$
By taking the expectation, we have
\begin{equation}\label{}
  \mathbb{E}\|\mathcal{G}q_\e f\|^2 \,=\,  \int_{(0,1)^3} G(x,y) G(x,z) \gamma_q\big(\frac{y-z}{\e}\big) f(y) f(z) \,dydzdx.
\end{equation}
Then, using estimate \eqref{boundG}, we have
\begin{eqnarray}
  \mathbb{E}\|\mathcal{G}q_\e f\|^2 &\leq& C\int_{(0,1)^3} \frac{1}{\vert y-x \vert^{1-\beta}} \frac{1}{\vert z-x \vert^{1-\beta}}\Big\vert \gamma_q\big(\frac{y-z}{\e}\big) f(y) f(z)\Big\vert \,dydzdx \label{ha1}\\
   &\leq& C \int_{(0,1)^2} \frac{1}{\vert y-z \vert^{1-2\beta}} \Big\vert \gamma_q\big(\frac{y-z}{\e}\big) f(y) f(z)\Big\vert \,dydz\label{ha2}\\
   &=& C \int_0^1\int_y^{y-1} \frac{1}{\vert z\vert^{1-2\beta}}\Big \vert \gamma_q\big(\frac{z}{\e}\big) f(y)f(y-z)\Big\vert \,dydz \notag \\
   &\leq& C\|f\|^2 \int_0^1 \frac{1}{\vert z\vert^{1-2\beta}} \big\vert \gamma_q\big(\frac{z}{\e}\big)\big\vert\,dz.\label{ha3}
\end{eqnarray}
To obtain the second inequality, we used \cite[Lemma A.1]{kk} which estimates the convolution of two potential functions and says that, for given two positive numbers $\alpha_1$ and $\alpha_2$ belonging to $(0,1)$ and two points $x\neq y$, we have
$$\int_0^1 \frac{1}{\vert z-x\vert^{\alpha_1}}\cdot \frac{1}{\vert z-y \vert^{\alpha_2}}\, dz\leq \begin{cases}
C \vert x-y\vert ^{1-(\alpha_1+\alpha_2)}, & \alpha_1+\alpha_2>1,\\
C (\log \vert x-y\vert +1), & \alpha_1+\alpha_2=1,\\
C  , & \alpha_1+\alpha_2<1.
\end{cases}$$
Now, we decompose the integration domain in \eqref{ha3} into two subdomains $D_{1, \e}$ and $D_{2,\e}$ as follows
\begin{equation}\label{}
  D_{1,\e} : = \big\{x\in(0,1),\quad \vert x \vert \leq M\e \big\} \qquad \text{and} \qquad D_{2,\e} = (0,1) \setminus D_{1,\e}.
  \end{equation}
The integration on $D_{1,\e}$ gives
\begin{equation}\label{aaa}\int_{D_{1,\e}}\frac{1}{\vert z\vert^{1-2\beta}} \vert \gamma_q\big(\frac{z}{\e}\big)\vert\,dz\leq \int_{D_{1,\e}}\frac{1}{\vert z\vert^{1-2\beta}}\,dz = \frac{M^{2\beta}}{2\beta} \e^{2\beta}.\end{equation}
Recall that the asymptotic behavior of the autocovariance function of $q$ implies the existence of some constant $C$ such that
\begin{align*}
\bv \gamma_q(x) \bv \leq C\, L( \vert x \vert )^{2m}  \vert x \vert^{-2m(1-H_0)}\,\,, \quad \forall\, x\in D_{2,\e}.
\end{align*}
Then, the integration over $D_{2,\e}$ is as follows
\begin{eqnarray*}
  \int_{D_{2,\e}}\frac{1}{\vert z\vert^{1-2\beta}} \vert \gamma_q\big(\frac{z}{\e}\big)\vert\,dz &\leq & C \int_{M\e}^1 \frac{1}{\vert z\vert^{1-2\beta}}L( \vert \frac{z}{\e} \vert )^{2m}  \vert \frac{z}{\e} \vert^{-2m(1-H_0)}\,dz \notag \\
   &=&  C \e^{2\beta}\int_M^{1/\e} \frac{L(|y|)^{2m}}{|y|^{1-2\beta+2m(1-H_0)}} \,dy\notag \\
   &\leq & C \e^{2\beta}\int_M^{1/\e} \frac{1}{|y|^{1-2\beta+2m(1-H_0)}} \,dy,\,\text{$L$ is assumed to be bounded}.\notag
\end{eqnarray*}
Where $2\beta=2m(1-H_0)$, the above integral equals $ C \e^{2\beta}(\log(\frac{1}{\e})-\log(M))$, and is of order $\varepsilon^{2m(1-H_0)}\vert \log \frac{1}{\varepsilon}\vert $. When $2\beta\neq 2m(1-H_0)$, the integral equals $C\varepsilon^{2\beta}(\varepsilon^{2m(1-H_0)-2\beta}-M^{2\beta-2m(1-H_0)})$.\\
Finally, combining the last estimates with \eqref{aaa} the lemma is proved.
\end{proof}
\begin{proof}[Proof of Theorem \ref{homogenization}]
The homogenized solution satisfies $$P(x,D)  u_0 \,=\,f.$$
We define $\chi_\e \,=\, -\mathcal{G} q_\e  u_0$. It is clear that $\chi_\e$ is the solution of
\begin{equation}\label{}
P(x,D) \chi_\e \,=\, -q_\e  u_0.
\end{equation}
Now, compare the two above equations with the one for $u_\e$, i.e. \eqref{EP1}. We get
\begin{equation}\label{}
  (P(x,D) +q_0 +q_\e) (u_\e-u_0-\chi_\e) \,=\, -q_\e\chi_\e,
\end{equation}
Since this equation is well-posed almost everywhere in $\Omega$ , we have $$u_\e-u_0 \,=\, \chi_\e-\mathcal{G}_\e q_\e \chi_\e.$$
This implies that:
\begin{equation}\label{}
  \|u_\e-u_0\| \,\leq\, \|\chi_\e\| + \|\mathcal{G}_\e\|_{\mathcal{L}(L^2)}\|q\|_\infty\|\chi_\e\|.
\end{equation}
Recall that the operator norm $\|\mathcal{G}_\e\|_{\mathcal{L}(L^2)}$ can be bounded uniformly in $\Omega$ and
\begin{equation}\label{}
  \|u_\e-u_0\| \,\leq\, C\|\chi_\e\|
\end{equation}
where $C$ depends on $\|q\|_\infty$ and the bound for $\|\mathcal{G}_\e\|_{\mathcal{L}}$
Finally, since $\chi_\e$ is of the form of $\mathcal{G}q_\e f$, we take the expectation and apply the previous lemma to complete the proof.
\end{proof}
Now, we will turn to the most important part of this work, in which, we will prove the convergence of the rescaled random fluctuations.
\subsection{Proof of theorem \ref{Non-central-limit-theorem}}
Before giving the proof of theorem \ref{Non-central-limit-theorem}, we recall some results that concern the convergence of random oscillatory integrals.

\begin{theorem}(\cite[Theorem 1.1]{lechiheb})\label{atef}
Let $g$ be the centered stationary Gaussian process defined by (\ref{g}), and assume that $\Phi\in L^2(\R,\nu)$ has Hermite rank $m\geq 1$.
Then, for any   $h\in C\big([0,1]\big)$, the following convergence in law takes place
\begin{equation}\label{goal}
 M^\e_h : =\frac{1}{ X(\e) }\int_0^1 \Phi[g(\frac{x}{\e})]  h(x)\,  dx \xrightarrow{\e\downarrow 0}   M^0_h := \frac{ V_m}{m!}  \int_0^1 h(x)\, dZ(x) \, ,
  \end{equation}
where $Z$ is the $m$th-Hermite process defined by \eqref{Hermite} and  $X(\e)$ is defined by \eqref{X}.
\end{theorem}
\begin{remark}\label{rem}
Clearly, the above result still holds true for any function h that is continuous except at finitely
many points.
\end{remark}
We can now proceed with the proof of Theorem \ref{Non-central-limit-theorem}.
\begin{proof}[Proof of Theorem \ref{Non-central-limit-theorem}]
The proof is divided into four steps.\\
{\bf (a) Preparation.}
The equation for $u_\e$ \eqref{EP1} may be formally recast as $$u_\e \,=\, \mathcal{G}(f-q_\e u_\e),$$ where $\mathcal{G}=(P(x,D))^{-1}$, and thus
\begin{equation}\label{}
  u_\e \,=\, \mathcal{G}f- \mathcal{G}q_\e\mathcal{G}f+\mathcal{G}q_\e\mathcal{G} q_\e u_\e.
\end{equation}
Because $ u_0=\mathcal{G}f$, we have
\begin{eqnarray*}
  u_\e- u_0 &=& - \mathcal{G}q_\e u_0+\mathcal{G}q_\e\mathcal{G} q_\e (u_\e- u_0+ u_0) \\
   &=&   - \mathcal{G}q_\e u_0+ \mathcal{G}q_\e\mathcal{G} q_\e u_0+ \mathcal{G}q_\e\mathcal{G} q_\e (u_\e- u_0).
\end{eqnarray*}
Then we can write
\begin{equation}\label{expression}
  \frac{u_\e(x)- u_0(x)}{X(\e)} \,=\, -\mathcal{I}_\e(x) \,+\, \underbrace{ \mathcal{Q}_\e(x) \,+\, r_\e(x)}_{=:\mathcal{R}^\e(x)} \,\, ,
\end{equation}
where
      $$\displaystyle  I_\e(x) \,=\, \frac{1}{X(\e)}\mathcal{G}q_\e u_0(x) \,=\, \frac{1}{X(\e)} \int_\mathbb{R} G(x,y)q_\e(y) u_0(y)\,dy,$$
\begin{eqnarray*}
    \mathcal{Q}_\e(x)&=& \frac{1}{X(\e)}\mathcal{G}q_\e\mathcal{G} q_\e u_0(x) \\
   &=&  \frac{1}{X(\e)} \int_{\mathbb{R}}\left(\int_\mathbb{R}G(x,y)q_\e(y)G(y,z) \,dy\right)\,q_\e(z) u_0(z)\,dz,
\end{eqnarray*}
 and
 \begin{eqnarray*}
   r_\e(x) &=& \frac{1}{X(\e)}\mathcal{G}q_\e\mathcal{G} q_\e (u_\e- u_0)(x) \\
    &=& \frac{1}{X(\e)} \int_{\mathbb{R}}\left(\int_\mathbb{R}G(x,y)q_\e(y)G(y,z)dy\right)q_\e(z)(u_\e- u_0)(z)dz.
 \end{eqnarray*}
\bigskip
Now, let us show the weak convergence of $I_\e(x)$ in $C([0,1])$.\\
\bigskip
{\bf (b) Convergence of $I_\e(x)$.}
In order to prove this claim, we start by establishing the f.d.d convergence and then we prove the tightness.
\begin{enumerate}[(i)]
\item {Convergence of finite dimensional distributions of $I_\e(x)$.} \quad  for any $x_1, x_2, \ldots, x_n\in\mathbb{R}$ and $\lambda_1, \lambda_2, ..., \lambda_n\in\mathbb{R}$ ($n\geq 1$), we have
\begin{eqnarray*}
  \sum_{k=1}^n \lambda_k \cdot I_\e(x_k)  &=& \sum_{k=1}^n \lambda_k \cdot  \frac{1}{X(\e)}\cdot \int_\mathbb{R} G(x_k, y) q^\e(y) u_0(y) \, dy  \\
   &=& \frac{1}{X(\e)}\cdot  \int_\mathbb{R}  \sum_{k=1}^n \lambda_k \cdot G(x_k, y) q^\e(y) u_0(y) \, dy.
\end{eqnarray*}
Note that the function $\sum_{k=1}^n \lambda_k G(x_k, \cdot)  u_0(\cdot)$ have at most finitely many discontinuities. Thus, Theorem \ref{atef} and Remark \ref{rem} imply that
    $\sum_{k=1}^n \lambda_k \, I_\e(x_k)$ converges in distribution to
    \begin{equation}\label{}
      \sum_{k=1}^n \lambda_k \cdot I(x_k) \,=\, \sum_{k=1}^n \lambda_k \cdot \frac{V_m}{m!}\int_\mathbb{R}  G(x_k, y)  u_0(y)\, dZ(y),
    \end{equation}
yielding the desired convergence of finite-dimensional distributions.
\item {Tightness of $I_\e(x)$.} \quad We check on Kolmogorov's criterion (\cite[Corollary 16.9]{OK02}). Fix $0\leq s\leq t\leq 1$. Then
\begin{eqnarray*}
&& \quad \E \big( \vert I_\e(s) - I_\e(t) \vert^2 \big)\\
 &=& \E \bigg[ \,\, \frac{1}{X(\e)^2} \bigg\vert \int_0^1  \big(G(s,y)-G(t,y)\big)q_\e(y)   u_0(y) \, dy  \, \bigg\vert^2 \, \bigg] \\
   &\leq& \frac{1}{X(\e)^2}  \int_0^1\int_0^1 \vert G(s,y)-G(t,y)\vert \cdot \vert G(s,\xi)-G(t,\xi)\vert \big\vert \gamma_q\Big( \frac{y-\xi}{\e} \Big) u_0(y) u_0(\xi)\big\vert\,dyd\xi\\
   &=& (\text{Lip}G)^2\vert s-t\vert^2\frac{1}{X(\e)^2} \int_0^1 \int_0^1 \big\vert \gamma_q\Big( \frac{y-\xi}{\e} \Big) u_0(y) u_0(\xi)\big\vert\,dyd\xi
\end{eqnarray*}
Now let us fix a number $\zeta\in(0,1)$, one can write (since $u_0$ is bounded)
\begin{eqnarray}
   && \quad \sup_{\e\in(0,\zeta)} \frac{1}{X(\e)^2} \int_0^1 \int_0^1 \big\vert \gamma_q\Big( \frac{\xi-y}{\e} \Big) u_0(y) u_0(\xi)\big\vert\,dyd\xi\nonumber\\
   &\leq & \text{Cst}\,\sup_{\e\in(0,\zeta)} \int_0^1\int_0^1 \big\vert \gamma_q\Big( \frac{\xi-y}{\e} \Big)\big\vert\,dyd\xi \label{aaaa}\\
   &\leq& \text{Cst}\,\sup_{\e\in(0,\zeta)} \left(\int_{D_{1,\e}}\big\vert \gamma_q\Big( \frac{\xi-y}{\e} \Big)\big\vert\,dyd\xi+ \int_{D_{2,\e}}\big\vert \gamma_q\Big( \frac{\xi-y}{\e} \Big)\big\vert\,dyd\xi\right)\nonumber
\end{eqnarray}
The integration over $D_{1,\e}$ gives
$$\int_{D_{1,\e}}\big\vert \gamma_q\Big( \frac{\xi-y}{\e} \Big)\big\vert\,dyd\xi\leq \text{Cst}.$$
On the other hand, by \eqref{correlation_particular},
$$\big\vert \gamma_q\Big( \frac{\xi-y}{\e} \Big)\big\vert \leq \text{Cst}\,L\Big(\frac{\xi-y}{\e}\Big) \Big\vert \frac{\xi-y}{\e}\Big\vert^{-2(1-H)}, \quad\forall (\xi,y)\in D_{2,\e}.$$
Thus, with $\beta>0$ small enough such that $2m\beta+2(1-H)\in(0,1)$, we have
\begin{eqnarray}
   && \quad \sup_{\e\in(0,\zeta)} \frac{1}{X(\e)^2} \int_0^1 \int_0^1 \big\vert \gamma_q\Big( \frac{\xi-y}{\e} \Big) u_0(y) u_0(\xi)\big\vert\,dyd\xi\nonumber\\
   &\leq& \text{Cst}\,\sup_{\e\in(0,\zeta)} \int_{D_{2,\e}}\left\{ \,\, \frac{ L\big( \vert (\xi-y)/\e \vert \big) }{ L(1/\e)} \, \right\} ^{2m}  \vert \xi-y\vert^{-2(1-H)}\,dyd\xi \nonumber\\
  &\leq & \text{Cst} \int_{D_{2,\e}} \vert \xi-y\vert^{-2m\beta -2(1-H)} \,dyd\xi\label{Ah}\\
  &\leq& \text{Cst},\nonumber
\end{eqnarray}
where \eqref{Ah} follows from Potter's theorem. Therefore,
\begin{align*}
\E \big( \vert I_\e(s) - I_\e(t) \vert^2 \big) \,\leq\, \text{Cst} (\text{Lip}G)^2\vert s-t\vert^2.
\end{align*}
This proves the tightness of $(I_\e(x))_\e$ by means of the usual Kolmogorov criterion.
\end{enumerate}
\bigskip
{\bf (c) Control on the remainder term $\mathcal{R}^\e(x)$ in \eqref{expression}.}  We shall prove that the process $\mathcal{R}^\e$ converges in probability to zero in $C([0,1])$. First we claim that if $G\in C([0,1])$, then there exists some constant $C = C(G)$ such that
 \begin{align} \label{claim0} \sup_{x\in[0,1]}  \E\left[\left(  \int_0^x q(\frac{y}{\e}) G(y) \, dy \right)^2\right] \leq  C\, X(\e)^2 \,\, . \end{align}
Indeed, the same argument we used for bounding \eqref{aaaa} works here as well:
\begin{align*}
 &\qquad \sup_{x\in[0,1]}  \E\left[\left( \int_0^x q(\frac{y}{\e}) G(y) \, dy \right)^2\right]  \\
 & \leq \| G \| _\infty^2  \int_{[0,1]^2} \bv \gamma_q( \vert \frac{y - z}{\e} \vert) \bv \, dy\, dz  \\
 & \leq   \| G \| _\infty^2   X(\e)^2 \left( \sup_{\e\in(0,\zeta)} \frac{1}{X(\e)^2} \int_{[0,1]^2} \bv \gamma_q( \vert \frac{y - z}{\e} \vert) \bv \, dy\, dz \right) \\
 & \leq \text{Cst} X(\e)^2 \,\, ,
\end{align*}
where the last inequality follows from \eqref{Ah}.\\
\textbf{Convergence of f.d.d. of $\mathcal{R}^\e(x)$.} This is divided into two steps \\
\textbf{Step 1:} Due to the explicit expression of $Q_\e(x)$, it follows that:
\begin{align}
&\qquad \E\big[ \vert Q_\e(x)\vert^2 \big] \notag\\
&= \mathbb{E}\left[ \Big\vert \frac{1}{X(\e)} \int_{\mathbb{R}}\left(\int_\mathbb{R}G(x,y)q_\e(y)G(y,z) \,dy\right)\,q_\e(z) u_0(z)\,dz \Big\vert^2 \right]\,\notag\\
& \leq \frac{1}{X(\e)^2} \int_{[0,1]^4} \big\vert G(x,y)G(y,z)G(x,\xi)G(\xi,\eta)\big\vert u_0(z)u_0(\eta)\mathbb{E}\Big[\big\vert q_\e(y)q_\e(z)
 q_\e(\xi)q_\e(\eta)\big\vert\Big] \,dydzd\xi d\eta \, \notag\\
  &\leq \text{Cst}\, X(\e)^{-2}\|G\|_\infty^4\|u_0\|^2_\infty \int_{[0,1]^4} \mathbb{E}\Big[\big\vert q_\e(y)q_\e(z) q_\e(\xi)q_\e(\eta)\big\vert\Big] \,dydzd\xi d\eta \,\, . \label{es}
 \end{align}
I order to estimate the fourth-order moments of $q(x,\omega)$  in \eqref{es}, we use \cite[Proposition 4.1]{Ahmad}. This Proposition states that for fixed set $F=\{1,2,3,4\}$, the collection of two pairs in $F$ can be defined by
 $$\mathrm{T}= \{p=\{(p(1),p(2)), (p(3),p(4))\}\,\, \text{such that}$$$$ p(i)\in F, p(1)\neq p(2), p(3)\neq p(4)\}.$$ Let $\mathrm{T}_\ast\subset \mathrm{T}$ be such that all $p(i)$ are different. Assume that $q=\{q(x)\}_{x\in\mathbb{R}_+}$ is constructed as in \eqref{q} so that $\{g(x\}_{x\in\mathbb{R}_+}$ is the Gaussian process given by \eqref{g} and the function $\Phi$ belongs to $\mathscr{G}_m$ and satisfies \eqref{AA1} and \eqref{3}. Then we have
   \begin{align}
 &\,\,\,\,\left\vert \mathbb{E}\prod_{i=1}^4 q(x_i)-\sum_{p\in \mathrm{T}_\ast} \gamma_q(x_{p(1)}-x_{p(2)})\gamma_q(x_{p(3)}-x_{p(4)})\right\vert \notag\\
 &\leq \text{Cst}\sum_{p\in\mathrm{T}\backslash\mathrm{T}_\ast}  \gamma_q(x_{p(1)}-x_{p(2)})\gamma_q(x_{p(3)}-x_{p(4)}),
\end{align}
where Cst is the one in \eqref{3}.\\
We return to $Q_\e(x)$ and by applying the above estimate, we get
\begin{equation}
  \E\big[ \vert Q_\e(x)\vert^2 \big] \leq  \text{Cst}\, X(\e)^{-2}\|G\|_\infty^4\|u_0\|^2_\infty \int_{[0,1]^4} \sum_p \Big\vert \gamma_q(\frac{x_{p(1)}-x_{p(2)}}{\e}) \gamma_q(\frac{x_{p(3)}-x_{p(4)}}{\e})\Big\vert\,dydzd\xi d\eta.
\end{equation}
Due to \eqref{correlation_particular}, $\gamma_q(\frac{x}{\e})$ is bounded by  $\text{Cst}\, L(|\frac{x}{\e}|)^{2m} \e^{2m(1-H_0)} |x|^{-2m(1-H_0)}$. Then each item in the sum has a contribution of size $\e^{2\times 2m(1-H_0)}$. Finally, due to the expression of $X(\e)$, we conclude that $ \E\big[ \vert Q_\e(x)\vert^2 \big] \leq \text{Cst} \e^{2m(1-H_0)}$, and this point proves the convergence of f.d.d. of $Q_\e(x)$ to the zero function.\\
 \textbf{Step 2:}\\
  Now, we will study the convergence of $r_\e(x)$. Due to the explicit expression of $r_\e(x)$, it follows from Cauchy-Schwarz inequality that
  {\footnotesize
  \begin{eqnarray}
    \mathbb{E} \vert r_\e(x) \vert &\leq& X(\e)^{-1}\mathbb{E}\Bigg[ \Big(\int_{[0,1]} \vert q_\e(z)(u_\e-u_0)(z) \vert^2 \,dz\Big)^{\frac{1}{2}}\Big(\int_{[0,1]}\Big(\int_{[0,1]}\vert G(x,y) q_\e(y) G(y,z) \vert\,dy\Big)^2 \,\,dz\Big)^{\frac{1}{2}}\Bigg] \nonumber \\
     &\leq & X(\e)^{-1} \|q\|_{\infty}  (\mathbb{E}\|u_\e-u_0\|^2)^{\frac{1}{2}}\Big(\mathbb{E} \int_{[0,1]^3}\Big\vert G(x,y)G(y,z)G(x,\xi)G(\xi,z)q_\e(y)q_\e(\xi) \Big\vert \,dyd\xi dz \Big)^{\frac{1}{2}}\nonumber\\
     &\leq& X(\e)^{-1} \|q\|_{\infty} \|G\|_\infty^4 (\mathbb{E}\|u_\e-u_0\|^2)^{\frac{1}{2}}\left(\int_{[0,1]^2} \vert \gamma_q\big(\frac{\xi-y}{\e}\big) \,d\xi dy\right)^{\frac{1}{2}}.\nonumber
  \end{eqnarray}} By the homogenization theorem \ref{homogenization}, the expectation of $\|u_\e-u_0\|^2$ is of size $\e^{2m(1-H_0)}$. The integral above can be bounded by $\text{Cst}\,X(\e)^2$ (see \eqref{claim0}). Finally, we have
 \begin{equation} \mathbb{E} \vert r_\e(x) \vert \leq \text{Cst}\,\, \e^{2m(1-H_0)}.\end{equation}
Combining the results of the two steps, the process $\mathcal{R}^\e$ converges in f.d.d. to the zero function.\\
\textbf{Tightness of $\mathcal{R}^\e(x)$.}\\
Fix $0\leq u < v\leq 1$. Then
{\footnotesize
\begin{eqnarray*}
 && \big\|  \mathcal{R}^\e(u) - \mathcal{R}^\e(v) \big\| ^2 \\
 &=& \mathbb{E} \Big[\frac{1}{X(\e)} \Big\vert \int_{[0,1]}\left(\int_{[0,1]} \big[ G(u,z)-G(v,z)\big] q_\e(z)G(z,\xi)\,  dz \right)q_\e(\xi) u_\e(\xi) \,d\xi \Big\vert\Big]^2 \\
   &\leq & \frac{1}{X(\e)^2}  \mathbb{E}\Big[ \|q\|_\infty^2 \|u_\e\|^2 \Big\vert\int_{[0,1]} \Big(\int_{[0,1]} \big[G(u,z)-G(v,z)\big]q_\e(z)G(z,\xi) \,d\xi\Big)^2\,dz\Big\vert\Big]\\
   &\leq & \frac{1}{X(\e)^2} \|q\|_\infty^2 \mathbb{E}\Big[ \|u_\e\|^2\Big\vert\int_{[0,1]^3}
\big[G(u,z)-G(v,z)\big]\big[G(u,\eta)-G(v,\eta)\big]q_\e(z)q_\e(\eta)G(z,\xi)G(\eta,\xi) \,dz d\eta d\xi\Big\vert\Big].
\end{eqnarray*}}
The fact that the operator norm of $\mathcal{G}_\e$ is bounded implies that $\|u_\e\|$ can be bounded uniformly with respect to $\omega$. We use the Lipschitz continuity and the uniform bound of $G$ to get
\begin{align*}
&\qquad \big\|  \mathcal{R}^\e(u) - \mathcal{R}^\e(v) \big\| ^2\\
& \leq \frac{1}{X(\e)^2} \|q\|_\infty^2 \text{Lip}(G)^2\|G\|^2_\infty |u-v|^2\int_{[0,1]^3} \Big\vert\gamma_q\big(\frac{z-\eta}{\e}\big)\Big\vert \,dz d\eta d\xi\\
& \leq \text{Cst}\, |u-v|^2,
\end{align*}
where the last inequality follows from the same arguments used for \eqref{claim0}. This completes the proof of the tightness of $\mathcal{R}^\e(x)$.\\
Combining the results above, the process $\mathcal{R}^\e$ converges in probability to the zero function.
\bigskip

{\bf (d) Conclusion}.
Combining the results of (a), (b) and (c),
the proof of Theorem \ref{Non-central-limit-theorem} is  concluded by evoking Slutsky's lemma.

\end{proof}

\bibliographystyle{plain}

\end{document}